\documentclass[10pt]{amsart}
\usepackage{amsfonts}

\usepackage{geometry}
\usepackage{amssymb,amsmath, amsthm}
\usepackage[numbers]{natbib}
\usepackage{graphicx}
\usepackage{braket}
\usepackage[pdftex,plainpages=false,colorlinks,hyperindex,bookmarksopen,linkcolor=red,citecolor=blue,urlcolor=blue]{hyperref}
\usepackage{mathrsfs}
\usepackage{xcolor}
\usepackage{epstopdf}
\usepackage{braket}
\usepackage{color}

\pdfoutput=1
\bibpunct{[}{]}{;}{n}{,}{,}
\newtheorem{theorem}{Theorem}[section]

\newtheorem{coro}{Corollary}[section]

\theoremstyle{remark}
\newtheorem{remark}{Remark}[section]

\newcommand{\be}{\begin{eqnarray}}
\newcommand{\ee}{\end{eqnarray}}

\newcommand{\B}{\mathfrak{B}}

\numberwithin{equation}{section}
\allowdisplaybreaks

 \title[Stochastic solutions for fractional complex heat equations]{%
Stochastic solutions for time-fractional heat equations with complex spatial
variables}
\author{Luisa Beghin}
\author{Alessandro De Gregorio}
\address{ Department of Statistical Sciences, Sapienza, University of Rome.
P.le Aldo Moro, 5, Rome, Italy}
\email{luisa.beghin@uniroma1.it}
\email{alessandro.degregorio@uniroma1.it}
\date{\today }
 
\begin{document}

 \begin{abstract}

We deal with complex spatial diffusion equations with
time-fractional derivative and study their stochastic solutions.
In particular, we complexify the integral operator solution to the
heat-type equation where the time derivative is replaced with the
convolution-type generalization of the regularized Caputo
derivative. We prove that this operator is solution of a complex
time-fractional heat equation with complex spatial variable. This
approach leads to a wrapped Brownian motion on a circle
time-changed by the inverse of the related subordinator. This
time-changed Brownian motion is analyzed and, in particular, some
results on its moments, as well as its construction as weak limit
of continuous-time random walks, are obtained. The extension of
our approach to the higher dimensional case is also provided.
 \medskip


 \smallskip


 \end{abstract}

 \maketitle

\section{Introduction} \label{sec1} 

The study of the evolution equations with complex spatial variables is a quite recent research topic
in the theory of the partial differential equations and complex analysis. In the pioneering works
\cite{gal}-\cite{gal1}, the authors proposed two different methods to complexify the spatial variable
appearing in different evolution equations (and keeping the time variable real). In particular,
for the heat equation a possible approach consists in the complexification of the spatial variable
in the linear semigroup operator
$$  
T_tf(x)=\frac{1}{\sqrt{2\pi t}}\int_{-\infty}^{+\infty} f(x-y)e^{-\frac{y^2}{2y}}dy
$$
representing the unique solution to the Cauchy problem
\begin{equation*}
\frac{\partial u}{\partial t}(x,t)= \frac{1}{2}\frac{\partial^2 u}{\partial
x^2}(x,t), \quad u(x,0)=f(x),
\end{equation*}
where $t>0,x\in \mathbb R$ and $f\in$ BUC$(\mathbb R)$.
This approach is interesting because, by exploiting the theory of semigroups of linear operators, it is possible to obtain a new complex version of the heat equation (see \eqref{eq:che} below) and study the properties of its analytic solution (see \cite{gal}-\cite{gal1}). Furthermore, it is worth to observe that in this framework a suitable probabilistic interpretation of the solution to \eqref{eq:che} leads to a wrapped up Brownian motion on a circle. The stochastic analysis of complex diffusion equations still seems to be an unexplored research topic.

The above mentioned theory was developed starting from partial
differential equations involving the standard time derivative.
The aim of this paper is to study complex versions of the
time-fractional heat equations obtained by 
complexifying 
the spatial variable only (and keeping the time variable real).
The main idea is to complexify the spatial variable in the
corresponding integral operator arising in the study of
time-fractional evolution equations. In particular, we study the
stochastic solution of the complex heat equation, when the
time-derivative is replaced by a convolution-type operator, which
generalizes the Caputo fractional derivative. We will adopt
the definition given in \cite{chen}, i.e.%
\begin{equation}
\mathfrak{D}_{t}^{g}u(t):=\frac{d}{dt}\int_{0}^{t}w(t-s)(u(s)-u(0))ds,\qquad
t\geq 0,  \label{chen}
\end{equation}%
where $w(\cdot )$ is the tail L\`{e}vy measure of a subordinator
 $\mathcal{H}_{g}:=\{\mathcal{H}_{g}(t)\}_{t\geq 0}$ and $g(\cdot )$ is its Laplace exponent, i.e.
$\mathbb{E} e^{-\theta \mathcal{H}_{g}(t)}=e^{-g(\theta )t}$, where $t,\theta
\geq 0$ (see Section \ref{sec3} for details on this definition).
The so-called generalized fractional calculus has been developed in recent years,
starting from Kochubei in \cite{koc}, by many authors (see, among the others,
\cite{toaldo}, \cite{giusti}, \cite{kol}). They extend the traditional
construct of fractional derivatives and integrals in order to
allow a wider class of kernels. Indeed, it is immediate to check
that, by choosing $w(t)=t^{-\alpha }/\Gamma (1-\alpha )$, for
$\alpha \in (0,1)$, which coincides with the tail L\'{e}vy measure
of an $\alpha $-stable subordinator $\mathcal{H}_{\alpha },$
the fractional derivative in (\ref{chen}) reduces to the so-called
``regularized Caputo derivative". Thus we will define the latter as
\begin{equation}
\frac{\partial ^{\alpha }}{\partial t^{\alpha
}}u(t):=\frac{1}{\Gamma (1-\alpha
)}\frac{d}{dt}\int_{0}^{t}(t-s)^{-\alpha }(u(s)-u(0))ds,\qquad
t\geq 0.  \label{cap}
\end{equation}%
It has been proved in \cite{meeJAP} that the solution to
\begin{equation*}
\frac{\partial ^{\alpha }}{\partial t^{\alpha }}u(x,t)=\Delta
u(x,t),\qquad t\geq 0,\text{ }x\in \mathbb{R}^{d},\text{ }d\geq 1,
\end{equation*}%
with $u(x,0)=f(x),$ admits the probabilistic representation
  $\mathbb{E}_{x}\left[ f(B(\mathcal{E}_{\alpha }(t)))\right]$,
  where $B:=\{B(t)\}_{t\geq 0} $, is the standard Brownian motion on $\mathbb{R}^{d}$,
with infinitesimal generator $\Delta$, and   $\mathcal{E}_{\alpha }:=\{\mathcal{E}_{\alpha }(t)\}_{t\geq 0}$, 
 is the inverse of the stable subordinator $\mathcal{H}_{\alpha }$ (independent of $B$). 

The previous result has been extended to the case of a strong Markov process
$X:=\{X(t)\}_{t\geq 0}$ (on a separable locally compact Hausdorff space $E$)
whose transition semigroup is a uniformly bounded strong continuous
semigroup in some Banach space and has infinitesimal generator $\mathcal{A}$.
In this case, the solution to %
\begin{equation*}
\left( \mathfrak{D}_{t}^{g}+b\frac{\partial }{\partial t}\right) u(x,t)=%
\mathcal{A}u(x,t),\qquad t\geq 0,\text{ }x\in E,\text{ }b\geq 0,
\end{equation*}%
with $u(x,0)=f(x),$ is represented by
  $\mathbb{E}_{x}\left[ f(X(\mathcal{E}_{g}(t)))\right]$, where $\mathcal{E}_{g}:=\{\mathcal{E}_{g}(t)\}_{t\geq 0}$,
  is the inverse of the general subordinator $\mathcal{H}_{g}$, with
drift $b,$ and is independent of $X$ (see \cite{chen}, for details).

If we denote the unit disk as  $D:=\{z\in \mathbb{C};|z|<1\}$ and
we consider the space $A(D)=\{f:\overline{D}\rightarrow :f$ is
analytic on $D$, continuous on $\overline{D}\}$, endowed with the
uniform norm, then we study here the solution to the following Cauchy problem
\begin{equation*}
\left( \mathfrak{D}_{t}^{g}+b\frac{\partial }{\partial t}\right) u(z,t)=%
\frac{1}{2}\frac{\partial ^{2}u}{\partial \varphi ^{2}}(z,t),\quad
(t,z)\in (0,\infty )\times D\setminus \{0\},\ z=re^{i\varphi },
\end{equation*}%
with $u(z,0)=f(z),$ where $f(z)\in A(D),$ $z\in \overline{D}$.

\smallskip 

The paper is organized as follows. Section \ref{sec2} contains the
probabilistic interpretation of the solution to the complex Cauchy
problem introduced in \cite{gal} and a discussion on the
properties of the circular Brownian motion. The generalized
fractional setting and the complex time-fractional heat equation
are introduced in Section \ref{sec3}, where the stochastic
solution related to the related complex Cauchy problem is
obtained. Section \ref{sec4} is devoted to the analysis of the
time-changed Brownian motion emerging in the previous section.
Some results on the moments of the process are provided, as well
as the construction of the time-changed process based on the
convergence of time-continuous random walks. Furthermore,
some special cases involving stable and tempered stable subordinators are examined.
The last section contains the analysis of the complex time-fractional heat equation
in higher dimensions.

\section{On the probabilistic meaning of heat-type equations with complex
space variables} \label{sec2}

\setcounter{section}{2} \setcounter{equation}{0} \setcounter{theorem}{0} 

Let $D:=\{z\in \mathbb{C};|z|<1\}$ be the open unit disk and introduce the Banach space $(A(D),\Vert \cdot \Vert ),$ where $A(D)=\{f: \overline D\to \mathbb C: f$ is analytic
on $D$, continuous on $\overline D\}$ , endowed with
the uniform norm $\Vert f\Vert =\sup \{|f|;z\in \overline{D}\}$. If $f\in A(D)$
then it can be represented in the series form $f(z)=\sum_{k=0}^{\infty
}a_{k}z^{k}$.

In the interesting paper \cite{gal}, it was proved (see Theorem 2.1 in \cite{gal}) that the singular (at $t=0$)
complex integral (that is known as a Gauss-Weierstrass integral)
\begin{equation}  \label{c0}
W_tf(z)= \frac{1}{\sqrt{2\pi t}} \int_{-\infty}^{+\infty}
f(ze^{-iu})e^{-u^2/2t}du,\quad t\geq 0,
\end{equation}
is a $C_0$-contraction semigroup of linear operators on $A(D).$ Furthermore, $u(z,t)=W_tf(z),$ is the unique solution (with $u(z,t)\in A(D)$  for a fixed $t$) for the Cauchy problem
\begin{equation}\label{eq:che}
\frac{\partial u}{\partial t}(z,t)= \frac{1}{2}\frac{\partial^2 u}{\partial
\varphi^2}(z,t), \quad z = re^{i\varphi}, 0<r<1, \varphi\in [0,2\pi),
\end{equation}
under the initial condition
\begin{equation}
u(z,0)= f(z), \quad z\in \overline D,
\end{equation}
where $f\in A(D).$

Here we briefly discuss the interesting probabilistic meaning of
representing the solution of a complexified Cauchy problem in this
way. Indeed, it is evident from \eqref{c0}, that the solution of
the latter can be expressed as
\begin{equation}
u(z,t) = \mathbb{E}f\left(ze^{-iB(t)}\right) = \mathbb{E}f\left(\mathfrak{B}%
_z(t)\right),
\end{equation}
where $B:=\{B(t)\}_{t\geq 0}$ is the $\mathbb R$-valued Brownian motion on $(\Omega,\mathcal F,\mathbb P)$ and $\mathfrak{B}%
_z(t)=ze^{-iB(t)}$. This means that the probabilistic representation of the
solution for this complexified Cauchy problem is directly related to a
circular or wrapped Brownian motion $\mathfrak{B}%
_z:=\{\mathfrak{B}_z(t)\}_{t\geq 0}$ moving on a circle with radius $r\in(0,1]$ (hereafter denoted by $\mathbb{S}_r$) and starting point $z$. Furthermore, $\mathfrak{B}$ stands for $\mathfrak{B}_1.$

For the sake of simplicity, we set $z=1.$ From the properties of the classical $\mathbb R$-valued Brownian motion, it is easy to characterized $\mathfrak{B}.$ Let $\overline{\mathfrak{B}}$ be the complex conjugate of $\mathfrak{B}.$ We observe that the wrapped Brownian motion $\mathfrak{B},$ satisfies the following properties:

1) $\mathfrak{B}(0)=1$ a.s.;

2) $\mathfrak{B}(t_k)\overline{\mathfrak{B}}(t_{k-1})$ with $k=1,2,...,n\in \mathbb N,$ $0=:t_0\leq t_1<t_2<...<t_n<\infty$ are independent;

3) $\mathfrak{B}(t)\overline{\mathfrak{B}}(s)$ has the same distribution of $\mathfrak{B}(t+h)\overline{\mathfrak{B}}(s+h),$ where $0\leq s<t, h\geq -s;$

4) for $0\leq s\leq t,$  $$\mathfrak{B}(t)\overline{\mathfrak{B}}(s)\sim WN(0,e^{-\frac{t-s}{2}} ),$$ where $WN(\mu, e^{-\frac{\sigma^2}{2}}),\mu\in \mathbb R, \sigma^2>0,$ stands for a wrapped normal random variables with probability density function given by
$$f(\varphi)=\frac{1}{\sqrt{2\pi}\sigma}\sum_{k=-\infty}^\infty e^{-\frac{(\varphi-\mu+2k\pi)^2}{2 \sigma^2}},\quad \varphi\in[0,2\pi).$$
This result follows by standard arguments on the wrapped distributions; i.e. by wrapping the $N(0,t-s)$ onto the circle (see, e.g., \cite{mardia}).

5) $\mathfrak{B}$ is a wrapped Gaussian process; i.e.  let $0=:t_0\leq t_1<t_2<...<t_n<\infty,$ the random vector
 $(\mathfrak{B}(t_1),\mathfrak{B}(t_2),...,\mathfrak{B}(t_n))$ is multivariate wrapped normal in the following sense
$$\prod_{k=1}^n(\mathfrak{B}(t_k))^{\alpha_k}, \quad \alpha_k\in\mathbb R,$$
admits a one-dimensional wrapped gaussian distribution. Indeed,
$$\prod_{k=1}^n(\mathfrak{B}(t_k))^{\alpha_k}=e^{i\sum_{k=1}^n\alpha_kB(t_k) }.$$
Since $B$ is a Gaussian process, it follows that
$\sum\limits_{k=1}^n \alpha_kB(t_k)\sim N(0,\sum\limits_{k=1}^n \alpha_k^2 t_k)$.
 Then, as in the previous point
$$\prod_{k=1}^n(\mathfrak{B}(t_k))^{\alpha_k}\sim WN(0, e^{-\frac{\sum_{k=1}^n \alpha_k^2 t_k}{2}}).$$

\section{Time-fractional diffusive-type equations with a complex spatial variable} \label{sec3} 

\setcounter{section}{3} \setcounter{equation}{0} \setcounter{theorem}{0}

Let us introduce a time-fractional version of the complex heat equation \eqref{eq:che} and study its stochastic solution.

Let $g:(0,+\infty )\rightarrow \mathbb{R}$ be a Bernstein function (i.e. a
non-negative, $C^{\infty }$ function such that $(-1)^{k-1}g^{(k)}(x)\leq 0$,
$\forall x>0$, $k\in \mathbb{N}$). Then, it is well-known that the following
representation holds (see e.g, \cite{schilling})
\begin{equation}
g(x)=a+bx+\int_{0}^{\infty }(1-e^{-sx})\nu (ds),\quad b\geq 0,  \label{gx}
\end{equation}%
where $\nu (\cdot )$ is a non-negative measure on $(0,+\infty )$, satisfying
the condition
\begin{equation}
\int_{0}^{\infty }(z\wedge 1)\nu (dz)<\infty ,  \notag
\end{equation}%
i.e. $\nu $ is a L\'{e}vy measure.

Let $w(s) = \int_s^{+\infty}\nu(dz)$ be its tail, in this paper we consider
the following convolution-type derivative (see \cite{chen})
\begin{equation}\label{eq:cd}
\mathfrak{D}_t^gu(t):=\frac{d}{dt}\int_0^t w(t-s)(u(s)-u(0))ds.
\end{equation}
Typically $w$ is a non-negative decreasing function on $(0,+\infty)$ that
blows up at $x = 0$ and locally integrable on $[0,\infty)$. We refer to \cite{chen} for the functional setting,
observing that obviously this definition is a generalization of the Caputo
fractional derivative (see, e.g., \cite{kilbas}): the latter is recovered, as a special case, for $%
w(s)= \frac{s^{-\alpha}}{\Gamma(1-\alpha)}$, with $\alpha \in
(0,1)$. Observe that we used a quite different notation from
\cite{chen} in order to underline the connection between this
generalized fractional derivative and the particular choice of the
underlying Bernstein function $g$. We remark that a similar
probabilistic approach to the generalized time-fractional
derivatives have been developed in \cite{toaldo}. It is similar
but not equivalent. Hereafter, we exclude compound Poisson
subordinator, namely we assume that $a=0$ and that the tail
measure $w(\cdot )$ is
infinite in the origin and absolutely continuous on $(0,+\infty )$. Let now $\mathcal H_g:=%
\{\mathcal{H}_{g}(t)\}_{t\geq 0}$ be the subordinator with L\'{e}vy measure $%
\nu $ and Laplace exponent $g$, i.e.
\begin{equation}
\mathbb{E}\left( e^{-\theta \mathcal{H}_{g}(t)}\right) =e^{-tg(\theta
)},\quad \theta \geq 0.  \label{be}
\end{equation}%
(see, e.g., \cite{app}).
We denote by $\mathcal{E}_{g}:=\{\mathcal{E}_{g}(t)\}_{t\geq 0}
$, the inverse (or hitting-time)
process $\mathcal E_g(t):=\inf\{s>0: \mathcal H_g(s)>t\}$, i.e.
\begin{equation}
\bigg\{\mathcal{E}_{g}(t)\geq s\bigg\}=\bigg\{\mathcal{H}_{g}(s)\leq t\bigg\}%
,\quad \forall s,t\in \mathbb{R}^{+}.
\end{equation}%
By the assumptions on $w$, the subordinator $t\mapsto \mathcal{H}_{g}(t)$ associated
to $g$ is strictly increasing a.s. As a consequence, its inverse $t\mapsto
\mathcal{E}_{g}(t)$ is continuous a.s. We recall that the
time-Laplace transform of the density of  $\mathcal{E}_{g}(t)$ denoted by $m_g(s,t):=\mathbb{P}\left( \mathcal{E}_{g}(t)\in ds\right) /ds$
reads
\begin{equation}\label{eq:lt}
\int_{0}^{\infty }e^{-\theta t}m_g(s,t) dt=\frac{g(\theta )}{\theta }e^{-sg(\theta )}, \quad \theta \geq 0,
\end{equation}%
see, for example, Proposition 3.2 in \cite{toaldo}.

We now consider the standard $\mathbb R$-valued Brownian motion $B$, time-changed by $%
\mathcal{E}_{g}(t), t\geq 0,$ (under the assumption that $B$ and $\mathcal{E}_{g}$
are mutually independent); i.e. $\{B(\mathcal{E}_{g}(t))\}_{t\geq 0}$. Then, the density of $B(\mathcal{E}_{g}(t))$ for a fixed $t> 0,$ is
given by
\begin{equation}
\ell _{g}(x,t)=\int_{0}^{+\infty }\frac{e^{-\frac{x^{2}}{2y}}}{\sqrt{2\pi y}}%
m_g(y,t)dy,\quad x\in \mathbb R.  \label{lg}
\end{equation}

For all $t> 0$, we can define on $D$, the following complex
integral
\begin{equation}\label{eq:c0tg}
W_{t}^{g} f(z):=\int_{\mathbb{R}}f(ze^{-iu})\ell _{g}(u,t)du,\quad z\in \overline D,
\end{equation}%
where $\ell_{g}$ is given in \eqref{lg}.  The following stochastic interpretation of \eqref{eq:c0tg} emerges
\begin{equation}\label{eq:c0tg2}
W_{t}^{g} f(z)=\mathbb{E}%
f(\B_g^z(t)),
\end{equation}%
where
\begin{equation}\label{eq:sucb}
\B_g^z:=\{\B_g^z(t)\}_{t\geq 0}:=\{ze^{-i B(\mathcal E_g(t))}\}_{t\geq 0},
\end{equation}
is the time-changed circular Brownian motion moving on a circle
with radius $r\in(0,1]$ with starting point $z,$ obtained from the
wrapped up process $\B^z$ introduced  in the previous section. We
observe that $$W_{t}^{g}f(z) =\mathbb E [W_{\mathcal
E_g(t)}f(z)]=\int_0^\infty W_y f(z)m_g(y,t)dy,$$ that is
$W_{t}^{g}$ arises by the time-change of the $C_0$-semigroup
\eqref{c0}.

We have the following analytic results concerning the convolution operator $W_t^g$.

\begin{theorem}\label{teofc}
(i) If $f\in A(D)$, then we have that for any $t>0,$ we have that $$W_{t}^{g}: A(D)\to A(D);$$ i.e. $W_{t}^{g}f(z)$ is
analytic in $D$
\begin{equation}
W_{t}^{g}f(z)=\sum_{k=0}^{\infty }a_{k}z^{k}d_{k}(t),  \label{gsol}
\end{equation}%
where $d_{k}(t):=\mathbb{E}[e^{-\frac{k^{2}}{2}\mathcal{E}_{g}(t)}],$ and if $f$ is continuous on $\overline{D},$ the integral $%
W_{t}^{g}f(z)$ is continuous on $\overline{D}$ as well.
Furthermore
$$R_\theta W_{t}^{g}f(z):=\int_0^\infty e^{-\theta t} W_{t}^{g}f(z) dt=\sum_{k=0}^{\infty }a_{k}z^{k}\tilde d_{k}(\theta),\quad \theta \geq 0, $$
where
\begin{equation}
\tilde{d_{k}}(\theta ):=\frac{g(\theta )/\theta }{g(\theta )+\frac{k^{2}}{%
2}}.
\end{equation}%

(ii) Moreover, $u(z,t)=W_t^g f(z)$ is the unique solution, belonging to $A(D)$ for any $t\geq 0,$ of the Cauchy problem
\begin{align}\label{cprob}
&\left(\mathfrak{D}_t^g+ b \frac{\partial}{\partial t}\right) u(z,t) = \frac{1%
}{2}\frac{\partial^2u}{\partial \varphi^2}(z,t),\\
 &\quad (t,z) \in (0,\infty)%
\times D\setminus\{0\}, \ z = r e^{i\varphi}, \notag
\end{align}
\begin{equation}
u(z,0) = f(z), \quad f(z)\in A(D), z \in \overline{D}.  \label{cprob1}
\end{equation}
\end{theorem}

\begin{proof}
(i) The representation \eqref{gsol} follows by considering that $f(z)\in
A(D) $ and by taking into account \eqref{lg} together with \eqref{eq:lt}.
Indeed, since $f(ze^{-iu})=\sum_{k=0}^{\infty }a_{k}z^{k}e^{-iuk},z\in D,$
is absolutely convergent, we can write
\begin{align*}
W_{t}^{g}f(z)& =\sum_{k=0}^{\infty }a_{k}z^{k} \mathbb E[e^{-i kB(\mathcal E_g(t))}]\\
&=\sum_{k=0}^{\infty }a_{k}z^{k}\int_{\mathbb{R}%
}e^{-iku}\ell _{g}(u,t)\text{d}u \\
& =\sum_{k=0}^{\infty }a_{k}z^{k}\int_{0}^{+\infty }\frac{1}{\sqrt{2\pi y}}%
m_g(y,t)dy \int_{\mathbb{R}}e^{-iku}\
e^{-\frac{u^{2}}{2y}}\text{d}u \\
& =\sum_{k=0}^{\infty }a_{k}z^{k}\int_{0}^{+\infty }e^{-\frac{k^{2}y}{2}}%
m_g(y,t)dy \\
& =\sum_{k=0}^{\infty }a_{k}z^{k}\mathbb{E}(e^{-\frac{k^{2}}{2}\mathcal{E}%
_{g}(t)})
\end{align*}%

By using the same arguments in \cite{gal1}, it is possible to prove that if $%
z_n,z_0\in \overline D,$ with $\lim_{n\to\infty}z_n=z_0,$ we get
\begin{align*}
|W_t^gf(z_n)-W_t^gf(z_0)|&\leq \int_{\mathbb{R}}\omega_1(f;
|z_n-z_0|)_{\overline D}\, \ell_g(u,t)\text{d} u \\
&=\omega_1(f; |z_n-z_0|)_{\overline D},
\end{align*}
where $\omega_1(f; \delta)_{\overline D}:=\sup\{|f(u)-f(v)|; |u-v|\leq
\delta, u,v\in \overline D\}$ is the modulus of continuity of $f.$ Then if $%
f $ is continuous on $\overline D,$ the integral $W_t^gf(z)$ is continuous
on $\overline D,$ as $n\to\infty.$

From \eqref{gsol} and by exploiting \eqref{eq:lt}, we obtain that
\begin{align*}
R_\theta W_{t}^{g}f(z)&=\sum_{k=0}^{\infty }a_{k}z^{k}\int_{0}^{\infty }e^{-\theta t}\mathbb{E}(e^{-\frac{k^{2}}{2}\mathcal{E}%
_{g}(t)})dt\\
&=\sum_{k=0}^{\infty }a_{k}z^{k}\frac{g(\theta )/\theta }{\frac{k^{2}}{2}+g(\theta
)}.
\end{align*}

(ii) From Theorem 2.1 in \cite{chen}, we can observe that: 1) $\mathfrak{D}_{t}^{g}W_{t}^{g}f(z)$ is well-defined since the integral appearing in the definition of $\mathfrak{D}_{t}^{g}$ is absolutely convergent in the Banach space $(A(D),||\cdot||);$ 2) for $b>0,$ $t\mapsto W_t^g f(\cdot)$ is globally Lipschitz in $(A(D),||\cdot||)$ and then $\frac{\partial}{\partial t}W_{t}^{g}f(\cdot)$ exits in $(A(D),||\cdot||)$ for a.s. $t\geq 0.$ In order to prove that \eqref{gsol} coincides with the solution of %
\eqref{cprob} with initial condition \eqref{cprob1} we take the time-Laplace
transform of both sides of \eqref{cprob}. By applying the result (2.18) in \cite{meetoa}, on the Laplace transform of the generalized derivative \eqref{eq:cd}, we obtain
\begin{align}
\mathcal{L}\left\{\mathfrak{D}_{t}^{g}W_{t}^{g}f(z)+b\frac{\partial}{\partial t}W_{t}^{g}f(z);\theta\right\} & =g(\theta
)R_\theta W_{t}^{g}f(z)-\frac{g(\theta )}{%
\theta }f(z)  \notag \\
& =g(\theta )\sum_{k=0}^{\infty }a_{k}z^{k}\tilde{d}_{k}(\theta )-\frac{%
g(\theta )}{\theta }f(z)  \notag \\
& =\frac{g(\theta )}{\theta }\bigg[\sum_{k=0}^{\infty }\frac{%
a_{k}z^{k}g(\theta )}{g(\theta )+\frac{k^{2}}{2}}-f(z)\bigg]  \label{a}
\end{align}%
and
\vskip -10pt 
\begin{align}
\frac{1}{2}\frac{\partial ^{2}}{\partial \varphi ^{2}}\mathcal{L}%
\{W_{t}^{g}f(z);\theta \}& =\frac{1}{2}\frac{\partial ^{2}}{\partial
\varphi ^{2}}\sum_{k=0}^{\infty }a_{k}z^{k}\tilde{d}_{k}(\theta )  \notag \\
& =-\frac{1}{2}\frac{g(\theta )}{\theta} %
\sum_{k=0}^{\infty }\frac{a_{k}k^{2}z^{k}}{\frac{k^{2}}{2}+g(\theta )}.  \label{b}
\end{align}%
Furthermore, from \eqref{gsol}, it is easy to prove that
\begin{align}\label{c}
W_{0}^{g}f(z)=f(z).
\end{align}%
The result follows by considering \eqref{a}, \eqref{b} and \eqref{c} together and
taking into account that $f(z)=\sum_{k=0}^{\infty }a_{k}z^{k}$.
\end{proof}

\begin{remark}
In Theorem \ref{teofc}, we proved that $u(t,z)=W_t^g f(z)$ is the classical solution with $u\in C^{\infty}([0,\infty); A(D)).$ Actually, it is also possible to prove that $u(t,z)$ is the unique strong solution of the fractional Cauchy problem \eqref{cprob}-\eqref{cprob1} (see Theorem 2.1 in \cite{chen} for the exact statement).
\end{remark}

\section{Time-changed wrapped Brownian motion}\label{sec4}

\setcounter{section}{4} \setcounter{equation}{0} \setcounter{theorem}{0} 

For simplicity and without loss of generality, we assume
throughout this section that $z=1$ and we consider the definition
$\mathfrak B_g(t):=e^{iB(\mathcal{E}_{g}(t))},\, t\geq 0$, which
is equivalent to (\ref{eq:sucb}). Then, by exploiting the theory
of wrapped distribution, for any $t>0$ we can write down its
probability density
$$\mu_{\mathfrak B_g}(\varphi,t):=\frac{\mathbb P(\B_g(t)\in d\varphi)}{d\varphi}=\sum_{k=-\infty}^\infty \ell_g(\varphi+2k\pi,t),\quad \varphi\in [0,2\pi).$$
The probability distribution on a circle is characterized by its
Fourier coefficients (see \cite{feller}, Theorem XIX 6.1). In our
case, for $k\in \mathbb N,$ we have that
\begin{align*}
\phi_k^t&:=\int_0^{2\pi}e^{ik \varphi}\mu_{\mathfrak B_g}(\varphi,t)d\varphi=\sum_{k=-\infty}^\infty \int_{2k\pi}^{2\pi(k+1)} e^{ik \varphi} \ell_g(\varphi, t)d\varphi\\
&=\mathbb E[e^{ikB( \mathcal{E}_g(t))}]=\mathbb E[e^{-\frac{ k^2\mathcal{E}_g(t)}{2}}]=d_k(t).
\end{align*}
Therefore,
\vskip -10pt 
$$\mu_{\mathfrak B_g}(\varphi)=\frac{1}{2\pi}\sum_{k=-\infty}^\infty \phi_k^t e^{-ik \varphi}
=\frac{1}{2\pi}\left (1+2\sum_{k=1}^\infty d_k(t)\cos(k\varphi)\right).$$

We now evaluate the Laplace transform of the first moments of
$\mathfrak{B}_{g}(t)$ by applying the results on the joint moments
of the inverse subordinators given in \cite{veill}. In particular,
we recall
that, for the Laplace transform of $K_{t_{1},...,t_{n}}(s_{1},...,s_{n}):=\mathbb{P}%
\left(
\mathcal{E}_{g}(t_{1})>s_{1},...,\mathcal{E}_{g}(t_{n})>s_{n}\right)
$, the following formula holds%
\begin{align} \label{veill}
\widetilde{K}_{\theta _{1},...,\theta
_{n}}(s_{1},...,s_{n})&:=\int_{0}^{\infty }...\int_{0}^{\infty
}e^{-\theta _{1}t_{1}...-\theta
_{n}t_{n}}K_{t_{1},...,t_{n}}(s_{1},...,s_{n})dt_{1}...dt_{n} \notag \\
&=\frac{1}{\prod\limits_{j=1}^{n}\theta _{j}}\exp \left\{
-\sum_{j=1}^{n}g\left( \sum_{k=j}^{n}\theta _{j(n)}\right) \left(
s_{j(i)}-s_{j(i-1)}\right) \right\} ,
\end{align}%
where $\theta_1,...,\theta_n>0,$   $0=s_{j(0)}\leq s_{j(1)}\leq ...\leq s_{j(n)}$ and
$j(1),...,j(n)$ is a permutation of the integers $1,...,n$ (by
convention $j(0)=0$).

\begin{theorem}
The Laplace transform of the moments of the process
$\mathfrak{B}_{g}$ is equal to
\begin{equation}
\mathcal{L}\left\{ \mathbb{E}\left[ \mathfrak{B}_{g}(t)\right]
^{r};\theta \right\} =\frac{2g(\theta )}{\theta (r+2g(\theta
))},\qquad \theta
>0,\;t\geq 0,\;r \in \mathbb{N}   \label{ee}
\end{equation}%
and%
\begin{align}\label{cov}
&\mathcal{L}\left\{
\mathbb{E}(\mathfrak{B}_{g}(t_{1})\mathfrak{B}_{g}(t_{2}));\theta
_{1},\theta _{2}\right\}\\
& =\frac{4g(\theta _{1})g(\theta
_{2})[g(\theta _{1}+\theta _{2})+2]+3[g(\theta _{1})+g(\theta
_{2})-g(\theta
_{1}+\theta _{2})]}{\theta _{1}\theta _{2}\left[ 2+g(\theta _{1}+\theta _{2})%
\right] (3+2g(\theta _{1}))(1+2g(\theta _{2}))},  \notag
\end{align}%
for \, $0\leq t_{1}<t_{2}$ and $\theta _{1},\theta _{2}>0.$
\end{theorem}

\begin{proof}
The $r$-th moment in \eqref{ee} can be easily obtained by a
conditioning
argument and by considering \eqref{be}:%
\begin{align*}
V(t)& :=\mathbb{E}\left[ e^{iB(\mathcal{E}_{g}(t))}\right] ^{r}=\mathbb{E}%
\left[ \mathbb{E}\left[ \left.
e^{irB(\mathcal{E}_{g}(t))}\right\vert
\mathcal{E}_{g}(t)\right] \right]  \\
&=\mathbb{E}e^{-\frac{r}{2}\mathcal{E}_{g}(t)}=\int_{0}^{+\infty }e^{-\frac{%
rs}{2}}m_{g}(s,t)ds . 
\end{align*}%
Then, by taking the time-Laplace transform and considering
\eqref{lt}, we
have that, for $\theta >0,$%
\begin{equation*}
\widetilde{V}(\theta )=\int_{0}^{+\infty }e^{-\theta
t}V(t)dt=\int_{0}^{+\infty }e^{-\frac{rs}{2}}\widetilde{m}_{g}(s,\theta )ds=%
\frac{g(\theta )}{\theta (r/2+g(\theta ))},
\end{equation*}%
where $\widetilde{m}_{g}(s,\theta ):=\int_{0}^{+\infty }e^{-\theta
t}m_{g}(s,t)dt$. In order to prove formula \eqref{cov}, we write
\begin{align}
V(t_{1},t_{2})&:=\mathbb{E}\left[
\mathfrak{B}_{g}(t_{1})\mathfrak{B}_{g}(t_{2})\right]   \label{ll} \\
&=\mathbb{E}e^{iB(\mathcal{E}_{g}(t_{1}))+iB(\mathcal{E}_{g}(t_{2}))}=%
\mathbb{E}\left[ \mathbb{E}\left[ \left. e^{iB(\mathcal{E}_{g}(t_{1}))+iB(%
\mathcal{E}_{g}(t_{2}))}\right\vert \mathcal{E}_{g}(t_{1}),\mathcal{E}%
_{g}(t_{1})\right] \right]   \notag \\
&=\mathbb{E}e^{-\frac{1}{2}[\mathcal{E}_{g}(t_{1})+\mathcal{E}%
_{g}(t_{2})+2\min
\{\mathcal{E}_{g}(t_{1}),\mathcal{E}_{g}(t_{2})\}]}. \notag
\end{align}%
We start by evaluating, for any $\eta _{1},\eta _{2}>0$
\begin{align*}
V_{\eta _{1},\eta _{2}}(t_{1},t_{2})&=\mathbb{E}e^{-\eta _{1}\mathcal{E}_{g}(t_{1})-\eta _{2}\mathcal{E}%
_{g}(t_{2})} \\
&=\int_{0}^{+\infty }\int_{0}^{+\infty }e^{-\eta
_{1}s_{1}-\eta
_{2}s_{2}}\frac{\partial ^{2}}{\partial s_{1}\partial s_{2}}%
K_{t_{1},t_{2}}(s_{1},s_{2})ds_{1}ds_{2} \\
&=[\text{by repeatedly integrating by parts}] \\
&=K_{t_{1},t_{2}}(0,0)-\eta _{1}\int_{0}^{+\infty }e^{-\eta
_{1}s_{1}}K_{t_{1},t_{2}}(s_{1},0)ds_{1}\\
&\quad-\eta
_{2}\int_{0}^{+\infty
}e^{-\eta _{2}s_{2}}K_{t_{1},t_{2}}(0,s_{2})ds_{2}+ \\
&\quad+\eta _{1}\eta _{2}\int_{0}^{+\infty }\int_{0}^{+\infty
}e^{-\eta _{1}s_{1}-\eta
_{2}s_{2}}K_{t_{1},t_{2}}(s_{1},s_{2})ds_{1}ds_{2}.
\end{align*}%
We now consider that $K_{t_{1},t_{2}}(0,0)=\mathbb{P}\left( \mathcal{E}%
_{g}(t_{1})>0,\mathcal{E}_{g}(t_{2})>0\right) =1$ and that $%
K_{t_{1},t_{2}}(s_{1},0)=\mathbb{P}\left( \mathcal{E}_{g}(t_{1})>s_{1}%
\right) $, so that we can write%
\begin{eqnarray*}
\int_{0}^{+\infty }e^{-\eta s_{1}}K_{t_{1},t_{2}}(s_{1},0)ds_{1}
&=&\int_{0}^{+\infty }e^{-\eta s_{1}}\mathbb{P}\left( \mathcal{E}%
_{g}(t_{1})>s_{1}\right) ds_{1} \\
&=&\frac{1}{\eta }\left[ 1-\widetilde{m}_{g}(\eta ,t_{1})\right] ,
\end{eqnarray*}%
and analogously for $K_{t_{1},t_{2}}(0,s_{2}).$ Therefore, we get%
\begin{align}
V_{\eta _{1},\eta _{2}}(t_{1},t_{2})&=\widetilde{m}_{g}\left( \eta
_{1},t_{1}\right) +\widetilde{m}_{g}\left( \eta _{2},t_{2}\right)
-1\notag\\
&\quad+\eta _{1}\eta _{2}\int_{0}^{+\infty }\int_{0}^{+\infty
}e^{-\eta _{1}s_{1}-\eta
_{2}s_{2}}K_{t_{1},t_{2}}(s_{1},s_{2})ds_{1}ds_{2}. \label{lt}
\end{align}%
In order to apply \eqref{veill}, we evaluate the Laplace transform of %
\eqref{lt}, with respect to the time variables:%
\begin{align}
\widetilde{V}_{\eta _{1},\eta _{2}}(\theta _{1},\theta _{2})
&:=\int_{0}^{+\infty }\int_{0}^{+\infty }e^{-\theta
_{1}t_{1}-\theta
_{2}t_{2}}V_{\eta _{1},\eta _{2}}(t_{1},t_{2})dt_{1}dt_{2}  \label{lt2} \\
&=\frac{1}{\theta _{2}}\widetilde{\widetilde{m}}_{g}\left( \eta
_{1},\theta _{1}\right) +\frac{1}{\theta
_{1}}\widetilde{\widetilde{m}}_{g}\left( \eta
_{2},\theta _{2}\right) -\frac{1}{\theta _{1}\theta _{2}} 
\notag \\
&\quad+\eta _{1}\eta _{2}\int_{0}^{+\infty }\int_{0}^{+\infty
}e^{-\eta _{1}s_{1}-\eta _{2}s_{2}}\widetilde{K}_{\theta
_{1},\theta _{2}}(s_{1},s_{2})ds_{1}ds_{2},  \notag
\end{align}%
where $\widetilde{\widetilde{m}}_{g}\left( \eta ,\theta \right)
:=\int_{0}^{+\infty }e^{-\eta s}\widetilde{m}_{g}(s,\theta )ds.$
On the
other hand, we can rewrite the last integral in \eqref{lt2}, by applying %
\eqref{veill}, as follows%
\begin{eqnarray*}
&&\frac{1}{\theta _{1}\theta _{2}}\int_{0}^{+\infty
}\int_{s_{1}}^{+\infty }e^{-\eta _{1}s_{1}-\eta
_{2}s_{2}}e^{-s_{1}[g(\theta _{1}+\theta
_{2})-g(\theta _{2})]-s_{2}g(\theta _{2})}ds_{1}ds_{2} 
\\
&&+\frac{1}{\theta _{1}\theta _{2}}\int_{0}^{+\infty
}\int_{0}^{s_{1}}e^{-\eta _{1}s_{1}-\eta
_{2}s_{2}}e^{-s_{2}[g(\theta
_{1}+\theta _{2})-g(\theta _{1})]-s_{1}g(\theta _{1})}ds_{1}ds_{2} \\
&=&\frac{1}{\theta _{1}\theta _{2}}\int_{0}^{+\infty
}\int_{s_{1}}^{+\infty }e^{-[\eta _{1}+g(\theta _{1}+\theta
_{2})-g(\theta _{2})]s_{1}-[\eta
_{2}+g(\theta _{2})]s_{2}}ds_{1}ds_{2} 
\\
&&+\frac{1}{\theta _{1}\theta _{2}}\int_{0}^{+\infty
}\int_{s_{2}}^{+\infty }e^{-[\eta _{1}+g(\theta _{1})]s_{1}-[\eta
_{2}+g(\theta _{1}+\theta
_{2})-g(\theta _{1})]s_{2}}ds_{1}ds_{2} \\
&=&\frac{1}{\theta _{1}\theta _{2}}
\left[ \int_{0}^{+\infty} \! 
\int_{s_{1}}^{+\infty} e^{-s_{1}A_{1}-s_{2}B_{2}}ds_{1}ds_{2} + \! 
\int_{0}^{+\infty} \! 
\int_{s_{2}}^{+\infty }e^{-s_{1}B_{1}-s_{2}A_{2}}ds_{1}ds_{2}\right]  \\
&=&\frac{1}{\theta _{1}\theta _{2}}\left[ \frac{1}{B_{2}(A_{1}+B_{2})}
 +\frac{1}{B_{1}(A_{2}+B_{1})}\right] ,
\end{eqnarray*}%
where we put $A_{i}:=\eta _{i}+g(\theta _{1}+\theta _{2})-g(\theta
_{j}),$ for $i,j=1,2$ and $i\neq j$, $B_{i}:=\eta _{i}+g(\theta
_{i}),$ for $i=1,2.$ We consider that
$A_{1}+B_{2}=A_{2}+B_{1}=\eta _{1}+\eta _{2}+g(\theta
_{1}+\theta _{2}),$ so that we can write \eqref{lt2}, by recalling \eqref{lt}%
, as%
\begin{align}
\widetilde{V}_{\eta _{1},\eta _{2}}(\theta _{1},\theta _{2})  \label{ut} &=\frac{1}{\theta _{1}\theta _{2}}\left[ \frac{g(\theta
_{1})}{\eta _{1}+g(\theta _{1})}+\frac{g(\theta _{2})}{\eta
_{2}+g(\theta _{2})}-1\right]
  \\
&\quad+\frac{\eta _{1}\eta _{2}}{\theta _{1}\theta _{2}}\frac{\eta
_{1}+\eta _{2}+g(\theta _{1})+g(\theta _{2})}{\left[ \eta
_{1}+\eta _{2}+g(\theta _{1}+\theta _{2})\right] \left[ \eta
_{1}+g(\theta _{1})\right] \left[ \eta _{2}+g(\theta _{2})\right]
}.  \notag
\end{align}%
By taking into account that $\mathcal{H}_{g}$ is a.s. increasing
and its
inverse $\mathcal{E}_{g}$ is a.s. non-decreasing$,$ so that
 $\min \{\mathcal{E}_{g}(t_{1}),\mathcal{E}_{g}(t_{2})\}=\mathcal{E}_{g}(t_{1})$
 a.s., for $t_{2}>t_{1}$, formula \eqref{cov} follows from \eqref{ut}, with
$\eta _{1}=3/2$ and $\eta _{2}=1/2$, after some algebraic  calculations.
\end{proof}

We are also able to give an integral representation for the mixed
moment $\mathbb E[\mathfrak B_g(t)\overline{\mathfrak B_g(s)}]$,
for $t,s\geq 0$.

\begin{theorem}
Let $U_{g}(\tau ):=\mathbb{E}[\mathcal{E}_{g}(\tau )]$, $\tau \geq
0$ and $\mathbb E[(\mathcal{E}_{g}(t))^{k}]<\infty, \, k\in \mathbb N$, 
then
\begin{equation}
\mathbb{E}[\mathfrak{B}_{g}(t)\overline{\mathfrak{B}_{g}(s)}]=\mathbb{E}%
\mathfrak{B}_{g}(t\vee s)+\frac{1}{2}\int_{0}^{t\wedge s}\mathbb{E}\mathfrak{%
B}_{g}(t\vee s-\tau )dU_{g}(\tau )  \label{bb}
\end{equation}%
for \, $t,s\geq 0$ and $t\neq s$, while $\mathbb{E}[\mathfrak{B}_{g}(t)%
\overline{\mathfrak{B}_{g}(t)}]=1$.
\end{theorem}

\begin{proof}
Let $t>s$ and  denoting $%
U_{g}(s,t;k,j):=\mathbb{E}[(\mathcal{E}_{g}(s))^{k}(\mathcal{E}_{g}(t))^{j}]$, $%
s,t\geq 0$, $k,j\in \mathbb{N},$ we can write down
\begin{align}
\mathbb{E}[\mathfrak{B}_{g}(t)\overline{\mathfrak{B}_{g}(s)}]& =\mathbb{E}%
[e^{i(B(\mathcal{E}_{g}(t))-B(\mathcal{E}_{g}(s)))}]  \label{eq: mixmom} \\
& =\mathbb{E}[[e^{i(B(\mathcal{E}_{g}(t))-B(\mathcal{E}_{g}(s)))}|\mathcal{E}%
_{g}(t),\mathcal{E}_{g}(s)]]  \notag \\
&
=\mathbb{E}[e^{-\frac{1}{2}(\mathcal{E}_{g}(t)-\mathcal{E}_{g}(s))}]
\notag \\
& =\sum_{k=0}^{\infty }\frac{1}{k!}\left( -\frac{1}{2}\right) ^{k}\mathbb{E}%
[(\mathcal{E}_{g}(t)-\mathcal{E}_{g}(s))^{k}]  \notag \\
& =\sum_{k=0}^{\infty }\frac{1}{k!}\left( -\frac{1}{2}\right)
^{k}\sum_{j=0}^{k}(-1)^{k-j}\binom{k}{j}\mathbb{E}\left[ (\mathcal{E}%
_{g}(t))^{j}(\mathcal{E}_{g}(s))^{k-j}\right]   \notag \\
& =\sum_{k=1}^{\infty }\frac{1}{k!}\left( -\frac{1}{2}\right)
^{k}\sum_{j=0}^{k-1}(-1)^{k-j}\binom{k}{j}U_{g}(s,t;k-j,j)\notag\\
&\quad+\sum_{k=0}^{%
\infty }\frac{1}{k!}\left( -\frac{1}{2}\right) ^{k}\mathbb{E}\left[ (%
\mathcal{E}_{g}(t))^{k}\right]   \notag \\
& =\sum_{k=1}^{\infty }\frac{1}{k!}\left( -\frac{1}{2}\right)
^{k}\sum_{j=0}^{k-1}(-1)^{k-j}\binom{k}{j}U_{g}(s,t;k-j,j)+\mathbb{E}e^{-%
\frac{1}{2}\mathcal{E}_{g}(t)},  \notag
\end{align}%
where we have singled out the term $j=k$ in the second summation,
since it must be treated separately (in view of its different
behavior for $s=0$). Now we use the recursive representation of
the moments given by Theorem 4.2,
\cite{veill}, so that (\ref{eq: mixmom}) can be rewritten as follows%
\begin{align}\label{ii}
\mathbb{E}[\mathfrak{B}_{g}(t)\overline{\mathfrak{B}_{g}(s)}]&=\sum_{k=1}^{\infty }\frac{1}{k!}\left( -\frac{1}{2}\right)
^{k}\sum_{j=0}^{k-1}(-1)^{k-j}\binom{k}{j}\\
&\quad \times\int_{0}^{s}(k-j)U_{g}(s-\tau
,t-\tau ;k-j-1,j)dU_{g}(\tau )\notag\\
&\quad+\sum_{k=1}^{\infty }\frac{1}{k!}\left( -\frac{1}{2}\right)
^{k}\sum_{j=0}^{k-1}(-1)^{k-j}\binom{k}{j}\notag\\
&\quad\times\int_{0}^{s}jU_{g}(s-\tau
,t-\tau
;k-j,j-1)dU_{g}(\tau )+\mathbb{E}\mathfrak{B}_{g}(t)  \notag \\
&=:I_{1}+I_{2}+\mathbb{E}\mathfrak{B}_{g}(t) . 
 \notag
\end{align}%
The first term in (\ref{ii}) can be treated as follows%
\begin{align*}
I_{1} &=\sum_{k=1}^{\infty }\frac{1}{k!}\left(
-\frac{1}{2}\right)
^{k}\sum_{j=0}^{k-1}(-1)^{k-j}\binom{k}{j}\\
&\quad\times\int_{0}^{s}(k-j)U_{g}(s-\tau
,t-\tau ;k-j-1,j)dU_{g}(\tau ) \notag \\
&=\sum_{k=1}^{\infty }\frac{1}{(k-1)!}\left( -\frac{1}{2}\right)
^{k}\sum_{j=0}^{k-1}(-1)^{k-j}\binom{k-1}{j}\\
&\quad\times\int_{0}^{s}\mathbb{E}\left[ (%
\mathcal{E}_{g}(t-\tau ))^{j}(\mathcal{E}_{g}(s-\tau
))^{k-1-j}\right]
dU_{g}(\tau )   \notag \\
&=\frac{1}{2}\sum_{l=0}^{\infty }\frac{1}{l!}\left(
-\frac{1}{2}\right)^{l} \sum_{j=0}^{l}(-1)^{l-j}
\binom{l}{j} \int_{0}^{s}\! 
\mathbb{E}\left[ (%
\mathcal{E}_{g}(t\!-\!\tau ))^{j}(\mathcal{E}_{g}(s\!-\!\tau))^{l-j}\right]
dU_{g}(\tau )  \notag \\
&=\frac{1}{2}\int_{0}^{s}\mathbb{E}\left[ \mathfrak{B}_{g}(t-\tau )%
\overline{\mathfrak{B}_{g}(s-\tau )}\right] dU_{g}(\tau ),  \notag
\end{align*}%
while the second one reads%
\begin{align*}
I_{2} &=\sum_{k=1}^{\infty }\frac{1}{k!}\left(
-\frac{1}{2}\right)
^{k}\sum_{j=0}^{k-1}(-1)^{k-j}\binom{k}{j}\\
&\quad\times\int_{0}^{s}jU_{g}(s-\tau
,t-\tau
;k-j,j-1)dU_{g}(\tau )  \notag \\
&=\sum_{k=2}^{\infty }\frac{1}{(k-1)!}\left( -\frac{1}{2}\right)
^{k}\sum_{j=1}^{k-1}(-1)^{k-j}\binom{k-1}{j-1}\\
&\quad\times\int_{0}^{s}\mathbb{E}\left[ (%
\mathcal{E}_{g}(t-\tau ))^{j-1}(\mathcal{E}_{g}(s-\tau
))^{k-j}\right]
dU_{g}(\tau )  \notag \\
&=\sum_{k=2}^{\infty }\frac{1}{(k-1)!}\left( -\frac{1}{2}\right)
^{k}\sum_{m=0}^{k-2}(-1)^{k-m-1}\binom{k-1}{m}\\
&\quad\times\int_{0}^{s}\mathbb{E}
 \left[ (\mathcal{E}_{g}(t-\tau ))^{m}(\mathcal{E}_{g}(s-\tau))^{k-m-1}\right]
dU_{g}(\tau )  \notag \\
&=-\frac{1}{2}\sum_{k=2}^{\infty }\frac{1}{(k-1)!}\left( -\frac{1}{2}%
\right) ^{k-1}\int_{0}^{s}\mathbb{E}\left[ \left( \mathcal{E}_{g}(t-\tau )-%
\mathcal{E}_{g}(s-\tau )\right) ^{k-1}\right] dU_{g}(\tau ) 
 \notag \\
&+\frac{1}{2}\sum_{k=2}^{\infty }\frac{1}{(k-1)!}\left(
-\frac{1}{2}\right)
^{k-1}\int_{0}^{s}\mathbb{E}\left[ \mathbb{(}\mathcal{E}_{g}(t-\tau ))^{k-1}%
\right] dU_{g}(\tau )  \notag \\
&=-\frac{1}{2}\int_{0}^{s}\mathbb{E}\left[ \mathfrak{B}_{g}(t-\tau )%
\overline{\mathfrak{B}_{g}(s-\tau )}\right] dU_{g}(\tau )+\frac{1}{2}%
\int_{0}^{s}\mathbb{E}\mathfrak{B}_{g}(t-\tau )dU_{g}(\tau ).
\end{align*}
The result follows by inserting $I_{1}$ and $I_{2}$ into
(\ref{ii}) and treating the case $s>t$ analogously.
\end{proof}

We now prove that the wrapped Brownian motion
$\mathfrak{B}_{g}$ can be obtained as a scaling limit of a
transformed continuous-time random walk on a circle.
Let us denote by $\overset{J_{1}}{\Longrightarrow }$ the convergence in the
$J_{1}$ topology and by $\overset{M_{1}}{\Longrightarrow }$ the
convergence in the $M_{1}$ topology in the Skorohod space
$D([0,T),\mathbb{R}^{d})$, for $T>0$ and $d=1,2,...$
(see \cite{whitt} and \cite{mee} for details on $J_{1}$ and $M_{1}$ topologies).

\begin{theorem}
Let $c>0$ and let $Y_{j}^{(c)},$ $j=1,2...$, be i.i.d. random
variables with finite moments and scale parameter $c.$ Let
moreover $J_{j}^{(c)},$ $j=1,2...
$, be i.i.d. random variables, independent of $Y_{j}^{(c)},$ for any $%
j=1,2...$ and for any $c>0$, and such that for $T^{(c)}(ct):=\sum_{j=1}^{[ct]}J_{j}^{(c)}$ the following convergence holds
$\{T^{(c)}(ct)\}_{t \geq 0} \overset{%
J_{1}}{\Rightarrow } \{\mathcal{H}_{g}(t)\}_{t \geq 0},$ as $c\rightarrow +\infty ,$ in $%
D([0,+\infty ),\mathbb{R}^{+}).$ Then%
\begin{equation}
\{e^{i\sum_{j=1}^{N_{t}^{(c)}}Y_{j}^{(c)}}\}_{t \geq 0}\overset{M_{1}}{\Longrightarrow }%
\{\mathfrak{B}_{g}(t)\}_{t \geq 0},\qquad c\rightarrow +\infty ,
\label{con}
\end{equation}
in $D([0,+\infty ),\mathbb{S}_1),$ where $N_{t}^{(c)}:=\max \{n\geq
0:T^{(c)}(n)\leq t\}.$
\end{theorem}

\begin{proof}
The convergence in (\ref{con}) follows by the application of
Theorem 2.1 and Corollary 2.4 in \cite{meer1}, in the special case
where $A(t)=B(t),$ $t\geq 0:$ indeed, let Disc$(x)$ be the set of
discontinuities of $x$, the assumption that Disc$(\left\{
A(t)\right\} _{t\geq 0})\cap$ Disc$(\left\{
\mathcal{H}_{g}(t)\right\} _{t\geq 0})=\emptyset $ a.s. is
automatically satisfied because, as well-known, it is always
possible to choose a version of Brownian motion such that its
trajectories are continuous with probability one. Moreover, the
L\'{e}vy measure of $\mathcal{H}_{g}$ is infinite on $[0,+\infty
)$ by assumption. By the independence of $J_{j}^{(c)} $and
$Y_{j}^{(c)},$ for any $j=1,2...$ and by the 
 functional central limit theorem, we have that%
\begin{equation*}
\left\{ \sum_{j=1}^{[ct]}Y_{j}^{(c)},T^{(c)}(ct)\right\} _{t\geq 0}\overset{%
J_{1}}{\Longrightarrow }\left\{
B(t),\mathcal{H}_{g}(t)\right\}_{t\geq 0} ,\qquad c\rightarrow
+\infty ,
\end{equation*}%
in the $J_{1}$ topology on $D([0,+\infty ),\mathbb{R}\times
\mathbb{R}^{+}).$ Therefore, by the above mentioned Theorem 2.1 in
\cite{meer1}, the following
convergence holds%
\begin{equation*}
\left\{ \sum_{j=1}^{N_{t}^{(c)}}Y_{j}^{(c)}\right\} _{t\geq 0}\overset{M_{1}}{\Longrightarrow }\left\{B(%
\mathcal{E}_{g}(t))\right\} _{t\geq 0},\qquad c\rightarrow +\infty
,
\end{equation*}%
in the $M_{1}$ topology on $D([0,+\infty ),\mathbb{R}).$ The
result finally follows by applying the continuous mapping theorem
to the function $\phi
(\cdot ):\mathbb{R}\rightarrow \mathbb{C}$ defined as $\phi (x)=e^{ix},$ $%
x\in \mathbb{R}$.
\end{proof}

We refer to \cite{meer1} for the description of some relevant
situations where this kind of convergence can be appropriately
applied.

\subsubsection{The stable case} 

For $g(\theta )=\theta ^{\alpha }$, $\alpha \in (0,1)$, formula \eqref{gx}
holds for $a=0$ and $b=\lim_{\theta \rightarrow +\infty }g(\theta
)/\theta =0$. Moreover, the process $\mathcal{E}_{g}(t)$ reduces to the
inverse of the $\alpha $-stable subordinator (see e.g. \cite{straka}) and
the operator $\mathfrak{D}_{t}^{g}$ coincides with the Caputo
time-fractional derivative of order $\alpha $, namely $\partial ^{\alpha
}/\partial t^{\alpha }$. In this case, if we denote $W_{t}^{g}$ as $W_{t}^{\alpha }$, we have that
\begin{equation}
W_{t}^{\alpha }(f)(z)=\frac{1}{t^{\alpha /2}}\int_{\mathbb{R}}f\left(
ze^{iu}\right) W_{-\frac{\alpha }{2},1-\frac{\alpha }{2}}\left( -\frac{|u|}{%
t^{\alpha /2}}\right) du,  \label{wri}
\end{equation}%
where
\begin{equation*}
W_{\beta ,\gamma }(x)=\sum_{k=0}^{\infty }\frac{x^{k}}{k!\Gamma (\beta
k+\gamma )},
\end{equation*}%
is the Wright function, defined for $\beta, \gamma, x \in \mathbb{C}$. The representation \eqref{wri} follows by the fact
that the fundamental solution of the time-fractional diffusion equation
\begin{equation}
\frac{\partial ^{\alpha }}{\partial t^{\alpha }}u(x,t)=\frac{1}{2}\frac{%
\partial ^{2}}{\partial x^{2}}u(x,t),
\end{equation}%
involving Caputo time-fractional derivatives of order $\alpha \in (0,1)$ is
given by
\begin{equation}
u(x,t)=\frac{1}{t^{\alpha /2}}W_{-\alpha /2,1-\alpha /2}\left( -\frac{|x|}{%
t^{\alpha /2}}\right) ,
\end{equation}%
see for example \cite{MMP}. 

Then, in this case, as a consequence of Theorem 3.1, 
we have the following result.

\begin{coro}
Let $E_{\alpha }(x):=\sum\limits_{j=0}^{\infty }\frac{x^{j}}{\Gamma (\alpha j+1)}$,
for $x,\alpha \in \mathbb{C}$, if $f(z)\in A(D)$, then we have on $D$ that
\begin{equation*}
W_{t}^{\alpha }f(z)=\mathbb{E}f\left( ze^{iB(\mathcal{E}_{\alpha
}(t))}\right) =\int_{\mathbb{R}}f\left( ze^{iu}\right)
\ell_{\alpha }(u,t)du=\sum_{k=0}^{\infty }a_{k}d_{k}(t)z^{k},
\end{equation*}%
where $d_{k}(t)=E_{\alpha }\left( -\frac{k^{2}t^{\alpha }}{2}\right)$,
 $\alpha \in (0,1)$, $\mathcal{E}_{\alpha }(t)$ is the inverse of the
 $\alpha $-stable subordinator and $\ell_{\alpha }$ is the probability
density of the time-changed Brownian motion
 $\mathfrak{B}_{\alpha }(t):=B(\mathcal{E}_{\alpha }(t))$.

Moreover, $W^{\alpha}_tf(z)$ is the unique solution $u(z,t)$ for the
fractional Cauchy problem
\vskip -10pt  
\begin{align}
& \frac{\partial ^{\alpha }u}{\partial t^{\alpha }}(z,t)=\frac{1}{2}\frac{%
\partial ^{2}u}{\partial \varphi ^{2}}(z,t),\\
& (z,t)\in \mathbb{R}%
^{+}\times D,\ z=re^{i\varphi },\ r\in (0,1),\ \varphi \in \lbrack 0,2\pi )\notag
\\
& u(z,0)=f(z),\quad z\in \overline{D},f\in A(D). \label{coro}
\end{align}
\end{coro}

\begin{remark}
Observing that, for $\alpha =1/2$ the following equality in distribution
holds
\begin{equation}
B(\mathcal{E}_{1/2}(t))\overset{d}{=}B_{1}(|B_{2}(t)|),
\end{equation}%
where $B_{1}$ and $B_{2}$ are independent, we have that
\begin{equation}
W_{t}^{1/2}f(z)=\mathbb{E}f\left( ze^{iB_{1}(|B_{2}(t)|)}\right)
\end{equation}%
coincides with the solution to \eqref{coro}, for $\alpha =1/2$.
Therefore, in this special case, we have an iterated Brownian motion on the
circle.
\end{remark}

\begin{remark}
In the stable case, i.e. for $g(\theta )=\theta ^{\alpha }$, the
inverse transform of the $r$-th moment given in \eqref{ee} can be
easily obtained and reads%
\begin{equation*}
\mathbb{E}\left[ \mathfrak{B}_{\alpha }(t)\right] ^{r}=E_{\alpha }
  \left( -\frac{r}{2}t^{\alpha }\right) ,\qquad r\in \mathbb{N}.
\end{equation*}
Moreover, we can
evaluate explicitly the mixed moment in (\ref{eq: mixmom}), by recalling that
 $U_{g}(\tau )=\mathbb{E}\mathcal{E}%
_{\alpha }(\tau )=\tau ^{\alpha }/\Gamma (\alpha +1)$ and that
 $\mathbb{E}\mathfrak{B}_{\alpha}(\tau )=E_{\alpha }(-\tau ^{\alpha }/2)$,
 so that, for $s<t$, we get%
\begin{eqnarray*}
\mathbb{E}[\mathfrak{B}_{\alpha}(t)\overline{\mathfrak{B}_{\alpha}(s)}]
&=&E_{\alpha
}\left( -\frac{t^{\alpha }}{2}\right) +\frac{\alpha }{2\Gamma (\alpha +1)}%
\int_{0}^{s}E_{\alpha }\left( -\frac{(t-\tau )^{\alpha
}}{2}\right) \tau
^{\alpha -1}d\tau \\
&=&E_{\alpha }\left( -\frac{t^{\alpha }}{2}\right) +\frac{t^{\alpha }}{%
2\Gamma (\alpha )}\int_{0}^{s/t}E_{\alpha }\left( -\frac{t^{\alpha
}(1-y)^{\alpha }}{2}\right) y^{\alpha -1}dy \\
&=&E_{\alpha }\left( -\frac{t^{\alpha }}{2}\right) +\frac{t^{\alpha }}{%
2\Gamma (\alpha )}\sum_{j=0}^{\infty }\frac{\left( -t^{\alpha }/2\right) ^{j}%
}{\Gamma (\alpha j+1)}B\left( \alpha j+1,\alpha ;s/t\right) ,
\end{eqnarray*}%
where $B(a,b;x):=\int_{0}^{x}z^{a-1}(1-z)^{b-1}dz$ is the
incomplete beta function. Since, for $s\rightarrow t$, the
previous expression reduces to one, we can write, for any $s,t\geq 0$,%
\begin{align}
\mathbb{E}[\mathfrak{B}_{\alpha}(t)\overline{\mathfrak{B}_{\alpha}(s)}]
&=E_{\alpha }\left( -\frac{(t\vee s)^{\alpha }}{2}\right)\notag
\\
&\quad+\frac{(t\vee s)^{\alpha }}{2\Gamma (\alpha )}\sum_{j=0}^{\infty
}\frac{\left( -(t\vee s)^{\alpha }/2\right) ^{j}}{\Gamma (\alpha
j+1)}B\left( \alpha j+1,\alpha ;(t\wedge s)/(t\vee s)\right) .
\notag
\end{align}%
Finally, it is easy to check that, for $\alpha =1$, (\ref{bb}) reduces to $%
e^{-[t\vee s-t\wedge s]/2}$ as it should be, since, in this case $\mathcal{E}%
_{\alpha }(t)=t$, a.s. for any $t\geq 0.$
\end{remark}

\subsubsection{The tempered stable case} 

For $g(\theta )=(\theta +\mu )^{\alpha }-\mu ^{\alpha }$, for $\theta
,\mu \geq 0$, $a=b=0$, and the process $\mathcal{E}_{g}$ reduces to the
inverse of the tempered stable subordinator (in the next
$\mathcal{E}_{T}$) and the operator $\mathfrak{D}_{t}^{g}$ coincides with
the tempered derivative (see e.g. \cite{Baumer1} and \cite{beghin}) denoted by
\begin{equation}
D_{t}^{\alpha ,\mu }f(t):=e^{-\mu t}\frac{\partial ^{\alpha }}{\partial t^{\alpha }}(e^{\mu t}f(t))-\mu
^{\alpha }f(t)=\left( \mu +\frac{d}{dt}\right) ^{\alpha }f(t).  \label{tem}
\end{equation}%
Since in this case the tail L\'{e}vy measure is given by
\begin{equation}
w(s)=\frac{\alpha \mu ^{\alpha }\Gamma (-\alpha ,\mu s)}{\Gamma (1-\alpha )},
\end{equation}%
where
\begin{equation*}
\Gamma (\rho ,x)=\int_{x}^{+\infty }e^{-\omega }\omega ^{\rho -1}d\omega ,
\end{equation*}%
is the upper incomplete gamma function. The tempered derivative \eqref{tem}
can be also expressed in a convolution form, as follows
\begin{equation}
D_{t}^{\alpha ,\mu }f(t)=\frac{\alpha \mu ^{\alpha }}{\Gamma (1-\alpha )}%
\int_{0}^{t}\frac{\partial }{\partial z}f(t-z)\Gamma (-\alpha ,\mu z)dz.
\end{equation}

\begin{coro}
Let $W_{t}^{g}$ be denoted as $W_{t}^{\alpha ,\mu }$,
for $g(\theta )=(\theta +\mu )^{\alpha }-\mu ^{\alpha }$, then the complex integral
\begin{equation}
W_{t}^{g}f(z)=\int_{\mathbb{R}}f\left( ze^{iu}\right) \ell_{\alpha
,\mu }(u,t)du=\sum_{k=0}^{\infty }a_{k}z^{k}d_{k}(t),
\end{equation}%
where $\ell_{\alpha ,\mu }$ is the probability density of the
time-changed Brownian motion $B(\mathcal{E}_{T}(t))$ and
\begin{equation}
d_{k}(t)=\frac{1}{\Gamma (\alpha )}\int_{0}^{t}\Gamma (\alpha ;\mu (t-z))%
\frac{e^{-\mu z}}{z}E_{\alpha ,0}\left( (\mu ^{\alpha }-\frac{k^{2}}{2}%
)z^{\alpha }\right) dz
\end{equation}%
is the unique solution of the fractional Cauchy problem
\begin{align}
& D_{t}^{\alpha ,\mu }u(z,t)=\frac{1}{2}\frac{\partial ^{2}u}{\partial
\varphi ^{2}}(z,t),\\
& (z,t)\in D\times \mathbb{R}^{+},\
, \ z=re^{i\varphi },\ r\in (0,1),\ \varphi \in \lbrack 0,2\pi )\notag
\\
& u(z,0)=f(z),\quad z\in \overline{D},f\in A(D).
\end{align}
\end{coro} 

\begin{proof}
The result follows by using Theorem 3.1 
and by inverting the Laplace transform
\begin{equation}
\tilde{d}_k(\theta) = \frac{(\theta+\mu)^\alpha-\mu^\alpha}{\theta[%
(\theta+\mu)^\alpha-\mu^\alpha+\frac{k^2}{2}]}.
\end{equation}
This can be done by applying the well-known formula of the Laplace transform
of the two-parameter Mittag-Leffler function (see e.g. \cite{fra}) and
recalling that (see \cite{beghin1})
\begin{equation}
\mathcal{L}\{\Gamma(\alpha; \mu x); \theta\} = \frac{(\mu+\theta)^\alpha-%
\mu^\alpha}{\theta(\mu+\theta)^\alpha}.
\end{equation}
\end{proof}

\begin{remark}
Note that, in the special case $\mu =0$, the expression of $d_{k}(t)$
reduces, for any $k\in \mathbb{N}$ and $t\geq 0$ to
\begin{equation}
d_{k}(t)=\int_{0}^{t}\frac{1}{z}E_{\alpha ,0}\left( -\frac{k^{2}}{2}%
z^{\alpha }\right) dz=E_{\alpha}\left( -\frac{k^{2}}{2}t^{\alpha
}\right) ,
\end{equation}%
which coincides with the stable case.
\end{remark}


\section{Higher dimensional extensions}\label{sec5} 

\setcounter{section}{5} \setcounter{equation}{0} \setcounter{theorem}{0} 

Following \cite{gal}-\cite{gal1}, we can also consider the $n$-dimensional fractional Cauchy problem extending \eqref{cprob}-\eqref{cprob1}. In this case, the probabilistic interpretation  of the solution for the complexified multidimensional fractional heat equation can be obtain by means of $n$-dimensional Brownian motions time changed with the
inverse of the subordinator $\mathcal{E}_{g}$.

 Let $%
z_{1},z_{2},\cdots ,z_{n}\in \overline{D}$ and $f(z_{1},z_{2},\cdots ,z_{n})\in
A(D^{n})$, which means that $f(z_1,z_2,...,z_n)$ belongs to $A(D)$ for each complex variable $z_1,z_2,...,z_n.$ Let us deal with the arguments developed in Section \ref{sec3} and extend them to the multidimensional case. Let us  introduce the following complex integral
\begin{align} \label{H}
& \overline{W}_{t}^{g}f(z_{1},z_{2},\cdots ,z_{n}) \\
& =\frac{1}{(2\pi )^{n/2}}\int_{-\infty }^{+\infty }\dots
\int_{-\infty }^{+\infty }f(z_{1}e^{-iu_{1}},\dots
z_{n}e^{-iu_{n}})du_1\cdots du_n
\notag \\
&\quad\times\int_{0}^{+\infty
}\frac{e^{-\frac{u_{1}^{2}+u_{2}^{2}\dots
+u_{n}^{2}}{2y}}}{y^{n/2}}m_g(y,t)dy\notag\\
& =\mathbb{E}f\left( z_{1}e^{-iB_{1}(\mathcal{E}_{g}(t))},z_{2}e^{-iB_{2}(%
\mathcal{E}_{g}(t))},\dots ,z_{n}e^{-iB_{n}(\mathcal{E}_{g}(t))}\right)
\notag\\
& =\mathbb{E}f\left(\mathfrak B_{1,g}^{z_1}(t),\mathfrak B_{2,g}^{z_2}(t),\dots ,\mathfrak B_{n,g}^{z_n}(t)\right) ,
\notag
\end{align}%
for $%
z_{1},z_{2},\cdots ,z_{n}\in \overline{D}.$ Furthermore, let $\mathfrak B_{k,g}^{z_k}:=\{z_{k}e^{-iB_{k}(\mathcal{E}_{g}(t))}\}_{t \geq 0}$ with $z_k=r_k e^{i\varphi_k}, r_k\in(0,1],\varphi_k\in [0,2\pi),$ and $\{B_k(t)\}_{t\geq0}$ independent standard Brownian motions, for $k=1,2,...,n.$ Therefore, the time-changed process
$$\mathfrak B_g^{z_1,...,z_n}:=\left\{\left(\mathfrak B_{1,g}^{z_1}(t),\mathfrak B_{2,g}^{z_2}(t),\dots ,\mathfrak B_{n,g}^{z_n}(t)\right)\right\}_{t\geq 0}$$
represents a $n$-dimensional wrapped Brownian motion on the $n$-dimensional circle $\mathbb S^n:={\mathbb S_{r_1}\times\cdots\times\mathbb S_{r_n}},$ where the coordinate processes are the circular Brownian motions \eqref{eq:sucb}  with random time $\mathcal E_g$ (and then they are not independent). Therefore, simple calculations show that $\mathfrak B_g^{z_1,...,z_n}$ has  covariance matrix  with entries
\begin{align*}
q_{i,j}&=\mathbb E[\mathfrak B_{i,g}^{z_i}(t)\overline{\mathfrak B_{j,g}^{z_j}(t)}]- \mathbb E[\mathfrak B_{i,g}^{z_i}(t)]\mathbb E[\overline{\mathfrak B_{j,g}^{z_j}(t)}] \\
&=
\begin{cases}
z_iz_j\{\mathbb E[e^{-2 \mathcal E_g(t)}]-(\mathbb E[e^{-\frac{ \mathcal E_g(t)}{2}}])^2\},& i\neq j,\\
z_i^2\{1-(\mathbb E[e^{-\frac{ \mathcal E_g(t)}{2}}])^2\},& i=j.
\end{cases}
\end{align*}

Now, we recall that, by the multivariate Taylor's expansion, we can write
\begin{equation*}
f(z_{1},...,z_{n})=\sum_{k_{1},...,k_{n}=0}^{\infty
}a_{k_{1},...,k_{n}}z_{1}^{k_{1}}\cdot \cdot \cdot z_{n}^{k_{n}}.
\end{equation*}%
Then, we have the following result representing the multidimensional version of Theorem \ref{teofc}.

\begin{theorem}
(i) If $f\in A(D^{n})$, then we have that the complex integral \eqref{H}
can be written as%
\begin{equation}
\overline{W}_{t}^{g}f(z_{1},z_{2},...
, z_{n})=\sum_{k_{1},...,k_{n}=0}^{\infty
}a_{k_{1},...,k_{n}}z_{1}^{k_{1}}\cdot \cdot \cdot
z_{n}^{k_{n}}d_{k_{1},...,k_{n}}(t),  \label{H2}
\end{equation}%
where $d_{k_{1},...,k_{n}}(t )= \mathbb E[e^{-\frac{\mathcal E_g(t)}{2}\sum_{j=1}^n k_j^2}].$ Furthermore
\begin{align}\label{eq:mtl}
R_\theta \overline{W}_{t}^{g}f(z_{1},z_{2},...
,z_{n})&:=\int_0^\infty e^{-\theta t}   \overline{W}_{t}^{g}f(z_{1},z_{2},...
,z_{n}) dt\\
&=\sum_{k_{1},...,k_{n}=0}^{\infty
}a_{k_{1},...,k_{n}}z_{1}^{k_{1}}\cdot \cdot \cdot
z_{n}^{k_{n}}\tilde{d}_{k_{1},...,k_{n}}(\theta),\notag
\end{align}
where
\begin{equation}
\widetilde{d}_{k_{1},...,k_{n}}(\theta )=\frac{g(\theta )/\theta }{%
g(\theta )+(k_{1}^{2}+...+k_{n}^{2})/2}.  \label{dk}
\end{equation}

(ii) Let $t>0,$ $f\in A(D^n)$ and $z_1,...,z_n\in \overline D.$ The integral operator \eqref{H} is the unique solution $$u(z_{1},z_{2},\dots ,z_{n},t)=\overline{W}_{t}^{g}f(z_{1},z_{2},...
, z_{n})$$
(that belongs to $A(D)$ for each complex variable $z_k$) of the fractional Cauchy problem
\begin{align}
&\mathfrak D_{t}^{g}u(z_{1},z_{2},\dots ,z_{n},t) =\frac12\left[ \frac{\partial ^{2}}{\partial
\varphi _{1}^{2}}+...+\frac{\partial ^{2}}{\partial \varphi _{n}^{2}}\right]%
u(z_{1},z_{2},\dots ,z_{n},t),    \label{d1}\\
&z_1=r_1e^{i\varphi_1},...,z_n=r_ne^{i\varphi_n}\in D\setminus{\{0\}},\notag \\
& u(z_{1},z_{2},\dots ,z_{n},0)=f(z_{1},z_{2},\dots ,z_{n}),
\end{align}%
where $\varphi_k$ is the principal value of $z_k$.
\end{theorem}

\begin{proof}
(i) The representation \eqref{H2} follows by writing \eqref{H} as follows%
\begin{align*}
&\overline{W}_{t}^{g}f(z_{1},z_{2},\cdots ,z_{n}) \\
&=\sum_{k_{1},...,k_{n}=0}^{\infty }a_{k_{1},...,k_{n}}z_{1}^{k_{1}}\cdot
\cdot \cdot z_{n}^{k_{n}}\int_{-\infty }^{+\infty }du_{n}\dots \int_{-\infty
}^{+\infty }du_{1}e^{-i\sum_{j=1}^n u_{j}k_{j}} \\
&\quad\times \int_{0}^{+\infty }\frac{e^{-\frac{u_{1}^{2}+u_{2}^{2}\dots
+u_{n}^{2}}{2y}}}{(2\pi y)^{n/2}}m_g(y,t)dy \\
&=\sum_{k_{1},...,k_{n}=0}^{\infty
}a_{k_{1},...,k_{n}}z_{1}^{k_{1}}\cdot
\cdot \cdot z_{n}^{k_{n}}\mathbb{E}[%
e^{-i\sum_{j=1}^n k_{j}B_{j}(\mathcal{E}_{g}(t))}]\\
&=\sum_{k_{1},...,k_{n}=0}^{\infty
}a_{k_{1},...,k_{n}}z_{1}^{k_{1}}\cdot
\cdot \cdot z_{n}^{k_{n}} \mathbb E[e^{-\frac{\mathcal E_g(t)}{2}\sum_{j=1}^n k_j^2}].
\end{align*}%
Since the time-Laplace transform of $\mathbb{E}%
e^{-i\sum_{j=1}^n k_{j}B_{j}(\mathcal{E}_{g}(t))}$
coincides with \eqref{dk} (see, for example, \cite{BR}), the result \eqref{eq:mtl} holds true.

(ii) Analogously to the one-dimensional case, we take the time-Laplace transform
of both sides of \eqref{d1}, as follows
\begin{align}
& \mathcal{L}\{D_{t}^{g}\overline{W}_{t}^{g}(f)(z_{1},...,z_{n});\theta \}
\notag \\
& =\frac{g(\theta )}{\theta }\bigg[\sum_{k_{1}=0,...,k_{n}=0}^{\infty }%
\frac{a_{k_{1},...k_{n}}z_{1}^{k_{1}},...,z_{n}^{k_{n}}g(\theta )}{%
g(\theta )+(k_{1}^{2},...,k_{n}^{2})/2}-f(z_{1},...,z_{n})\bigg]  \label{d2}
\end{align}%
and
\begin{align}\label{d3}
&\frac{1}{2}\left[ \frac{\partial ^{2}}{\partial \varphi _{1}^{2}}+...+%
\frac{\partial ^{2}}{\partial \varphi _{n}^{2}}\right] \mathcal{L}\{%
\overline{W}_{t}^{g}(f)(z_{1},...,z_{n});\theta \} \\
&=-\frac{1}{2}\sum_{k_{1}=0,...,k_{n}=0}^{\infty }\frac{g(\theta )/\theta
}{\frac{(k_{1}^{2}+...+k_{n}^{2})}{2}+g(\theta )}%
a_{k_{1},...,k_{n}}(k_{1}^{2}+...+k_{n}^{2})z_{1}^{k_{1}}\cdot \cdot \cdot
z_{n}^{k_{n}}.  \notag  \label{bb}
\end{align}
Therefore \eqref{d2}, \eqref{d3} and \eqref{H2} allow to conclude the proof.
\end{proof}



\end{document}